\documentclass[listof=flat]{scrartcl}

\usepackage{graphicx}
\usepackage{color, colortbl}
\usepackage{multirow}
\usepackage{caption}
\usepackage{tabularx}
\usepackage{aligned-overset}
\usepackage{amssymb,amsmath,amsthm,mathtools, bbm, mathrsfs} 
\usepackage{nicefrac}
\usepackage[margin=1.25in]{geometry}
\usepackage[utf8]{inputenc}
\usepackage[ngerman,english]{babel}
\usepackage[T1]{fontenc}
\usepackage{bbm}
\usepackage{tabularx}
\usepackage{mathdots}
\usepackage[makeroom]{cancel}
\usepackage{tikz}
\usepackage[draft]{todonotes}
\usepackage{tikz-imagelabels}
\usepackage{caption}
\usepackage[labelformat=simple]{subcaption}

\usepackage[wby]{callouts}
\usepackage{hyperref}
\usepackage{varioref}
\usepackage{tikz}
\usetikzlibrary{patterns, decorations.pathmorphing} 

\DeclareFontEncoding{LS1}{}{}
\DeclareFontSubstitution{LS1}{stix}{m}{n}
\DeclareSymbolFont{stixletters}{LS1}{stix}{m}{it}
\DeclareMathAccent{\cev}{\mathord}{stixletters}{"91}
\newtheorem{theorem}{Theorem}
\newtheorem{proposition}{Proposition}

\newtheorem{lemma}{Lemma}[section]

\newtheorem{rem}{Remark}[section]

\definecolor{ao}{rgb}{0.0, 0.5, 0.0}
\definecolor{tuscanred}{rgb}{0.51, 0.21, 0.21}
\definecolor{babypink}{rgb}{0.96, 0.76, 0.76}
\definecolor{candypink}{rgb}{1.0, 0.67, 0.79}
\definecolor{darkpastelgreen}{rgb}{0.01, 0.75, 0.24}
\definecolor{red(pigment)}{rgb}{0.93, 0.11, 0.14}
\definecolor{unmellowyellow}{rgb}{1.0, 1.0, 0.5}
\definecolor{turquoiseblue}{rgb}{0.0, 1.0, 0.94}
\definecolor{blue(pigment)}{rgb}{0.23, 0.3, 0.92}

\newcommand\myworries[1]{\textcolor{red}{#1}}
\usepackage{physics}
\usepackage{bbm}
\usepackage[normalem]{ulem}
\usepackage{verbatim} 
\usepackage{algorithm}
\usepackage{algpseudocode}
\usepackage{float}

\newcommand{\e}{\mathrm{e}}

\begin{document}
	
	\title{The largest subcritical component in inhomogeneous random graphs of preferential attachment type}
	
	\author{Peter {M\"orters} and Nick Schleicher\\[-1mm]
    \footnotesize{University of Cologne}\\[-2mm] \footnotesize{Department of Mathematics}\\[-2mm] \footnotesize{Weyertal 86-90}\\[-2mm]  \footnotesize{50931 K\"oln, Germany}}
 
    \date{}
	
	\maketitle

 \vspace{-1cm}

	\begin{abstract}
\noindent\textbf{Abstract:} We identify the size of the largest connected component in a subcritical inhomogeneous random graph with a kernel of preferential attachment type. The component is polynomial in the graph size with an explicitly given exponent, which is strictly larger than the exponent for the largest degree in the graph. This is in stark contrast to the behaviour of inhomogeneous random graphs with a kernel of rank one. Our proof uses local approximation by branching random walks going well beyond the weak local limit and novel results on subcritical killed branching random walks. 
\end{abstract}


    \section{Introduction and main results}

Preferential attachment models give a credible explanation how typical features of networks, like scale-free degree distributions and small diameter, arise naturally from the basic construction principle of reinforcement. This makes them a popular model for scale-free random graphs. Unfortunately, the mathematical analysis of preferential attachment networks is much more challenging than that of many other scale-free network models, for example the configuration model. In particular, the problem of the size of the largest subcritical connected component, solved for the configuration model by Janson~\cite{janson}, is open for all model variants of preferential attachment. The purpose of the present paper is to solve this problem for a simplified class of network models of preferential attachment type. We believe that our model, which is an inhomogeneous random graph with a suitably chosen kernel,  has sufficiently many common features with the most studied models of preferential attachment networks to serve as a solvable model in this  universality class. Since inhomogeneous random graphs are interesting models in their own right, see \cite{Bollob_s_2007}, their analysis is also of independent interest.%
\bigskip\pagebreak[3]

The class of \emph{inhomogeneous random graphs} is parametrised by a symmetric kernel
$$\kappa\colon (0,1] \times (0,1]  \rightarrow (0, \infty)$$
and constructed such that, for each $n \in \mathbb{N}$, the graph $\mathscr G_n$ has vertex set $V_n = \{1, \ldots, n\}$ and 
edge set $E_n$ containing
each unordered pair of distinct vertices $\{i,j\}\subset V_n$ independently with probability
$$p_{ij}^{(n)}= \frac1n \kappa\left( \frac{i}{n}, \frac{j}{n} \right) \wedge 1.$$
Our idea is now to choose the kernel $\kappa$ in such a way that the inhomogeneous random graphs mimic the behaviour of preferential attachment models. In preferential attachment models vertices arrive one-by-one and attach themselves to earlier vertices with a probability proportional to their degree. Typically 
degrees grow polynomially so that, for some $\gamma>0$, the degree of vertex~$i$ at time~$j>i$ is of order $(j/i)^\gamma$. For the expected degree of vertex $j$ at its arrival time to remain bounded from zero and infinity we need that $\gamma<1$ and the proportionality factor in the connection probability to be of order \smash{$(\sum_{i=1}^{j-1} (j/i)^\gamma)^{-1}\approx (1/j)$}. Hence in the preferential attachment models for vertices with index $i<j$
we have connection probability $p_{ij}^{_{(n)}} \approx i^{-\gamma}j^{\gamma-1}$. To get the same connection probabilities in the inhomogeneous random graph we choose the kernel 
$$\kappa(x, y) = \beta (x \vee y)^{\gamma - 1} (x \wedge y)^{-\gamma},$$
where the parameter $0 < \gamma < 1$ controls the strength of the preferential attachment, and $\beta > 0$ is an edge density parameter. Note that $\kappa$ is homogeneous of index $-1$ and  therefore the resulting edge probability $p_{ij}^{_{(n)}}$ is independent of the graph size~$n$.
We refer to this model as the \emph{inhomogeneous random graph of preferential attachment type}. \bigskip

It is easy to see that the inhomogeneous random graph of preferential attachment type has an asymptotic degree distribution which is heavy-tailed with power-law exponent $\tau=1+\frac1\gamma$. The analysis of a preferential attachment model in~\cite{Dereich_2013} can be simplified, see~\cite{moerters} for details, and shows that the size $S_n^{\mathrm{max}}$ of the largest connected component in $\mathscr G_n$
satisfies
$$\lim_{n\to\infty} \frac{S_n^{\mathrm{max}}}{n}
= \theta(\beta) \left\{ 
\begin{array}{ll}
>0 & \text{if }
\beta > \beta_c := (\frac{1}{4}-\frac{\gamma}{2}) \vee 0,\\
=0 & \text{otherwise.}\\
\end{array}
\right. $$
In this paper we are interested in the \emph{subcritical regime}, i.e.\ we always assume that $\gamma<\frac12$ and $0<\beta<\beta_c$. In this case all component sizes are of smaller order than $n$. Our first result identifies the component sizes of vertices in a moving observation window. We say that vertex $o_n\in V_n$ is \emph{typical} if
$\frac{o_n}n \to u$ for some $u>0$ and the behaviour of \emph{early typical vertices} refers to features of $\mathscr G_n$ rooted in vertex $o_n$ that hold asymptotically as $u\downarrow0$. 
\smallskip

\begin{theorem}[Early typical vertices]\label{mainresult}
Let $S_n(i)$ be the size of the connected component of vertex $i\in V_n$ in the
inhomogeneous random graph of preferential attachment type in the subcritical regime.
If $o_n\in V_n$ is such that $\frac{o_n}n\to u\in(0,1]$, then
$$        \lim_{u \downarrow 0} \lim_{n \rightarrow \infty} \mathbb{P}\left( S_n(o_n) \geq u^{-\rho_-}x   \right)= \mathbb{P} \left( Y\geq x  \right) \, ,$$
for all $x>0$, where $$\rho_{\pm} = \tfrac{1}{2} \pm \sqrt{(\gamma-\tfrac{1}{2})^2+\beta(2\gamma-1)}.$$
and $Y$ is a positive random variable satisfying 
$$\mathbb{P}\left( Y \geq x \right) = x^{-(\rho_+/\rho_-) + o(1)} \text{ as $x\to\infty$.}$$
\end{theorem}

Our second theorem identifies the the size of the component of untypically early vertices. Here a vertex $o_n\in V_n$ is called \emph{untypically early} if 
$\frac{o_n}n \to 0$. 

\begin{theorem}[Untypically early vertices]\label{2}
Let $o_n\in V_n$ be such that $$ {o_n\to\infty} \text{ and } { \frac{o_n}{n} \to 0}.$$ Denoting 
by $S_n(o_n)$ the size of the component of $o_n$ in $\mathscr G_n$, then for any $\epsilon>0$ we have
$$     {\lim_{n \rightarrow \infty} \mathbb{P}\left( S_n(o_n) \geq (n/o_n)^{\rho_--\epsilon}   \right)= 1}.$$
\end{theorem}

The idea behind this result is to exploit a self-similarity feature of graphs of preferential attachment type and leverage~Theorem~\ref{mainresult}. Loosely speaking, we find for fixed small $u>0$ a positive  integer $k$ with $o_n\approx u^kn$. Then $o_n$ is early typical in the graph $\mathscr G_{u^{k-1}n}$ and by Theorem~\ref{mainresult} we get a connected component with size of order $u^{-\rho_-}$. Many vertices in this component are themselves early typical in $\mathscr G_{u^{k-2}n}$ and we can use 
Theorem~\ref{mainresult} again, getting a component with size of order $u^{-2\rho_-}$. Continuing the procedure altogether $k$ times we build a component of size
$u^{-k\rho_-}\approx (n/o_n)^{\rho_-}$ in $\mathscr G_{n}$.
\medskip

Theorem~2 gives a lower bound on the size of the largest component. As it describes the size of the components of the most powerful vertices in $\mathscr G_n$  it is plausible that this result also gives the right order of the largest component. Our main result confirms this. It
is the first result in the mathematical literature identifying the size of the largest subcritical component up to polynomial order for a random graph model of preferential attachment type.

\begin{theorem}[Largest subcritical component]
Denoting by $S_n^{\mathrm{max}}$ the size of the largest component in $\mathscr G_n$ we have
$$\lim_{n \rightarrow \infty} \frac{\log S_n^{\mathrm{max}}}{\log n} = \rho_{-},$$
in probability, where 
$$\rho_{-} = \tfrac{1}{2} - \sqrt{(\tfrac{1}{2}-\gamma)^2-\beta(1-2\gamma)} > \gamma.$$
\end{theorem}
\medskip

\begin{rem}
Observe that the size of the largest component in a finite random graph is bounded from below by the maximum over all degrees. In scale-free graphs this is of polynomial order in the graph size. It is shown in~\cite{janson} that this lower bound is sharp for configuration models and inhomogeneous random graphs with a kernel of rank one. In our model the largest degree is $n^{\gamma+o(1)}$, whereas the largest component has size $n^{\rho_-+o(1)}$ and is therefore much larger.  A lower bound on the largest component larger than the maximal degree has also been found for a different preferential attachment model in \cite{Ray}. 
As this effect is due to the self-similar nature of the graphs of preferential attachment type 
we conjecture that it is a universal feature of preferential attachment graphs that if the largest degree is 
\smash{$n^{\gamma(\beta)+o(1)}$}
the largest subcritical component is of 
size \smash{$n^{\rho(\beta)+o(1)}$} for some
\smash{$\rho(\beta)>\gamma(\beta)$} with 
\smash{$\rho(\beta)\to\frac12$} as $\beta\uparrow \beta_c$.
\end{rem}

The remainder of the paper is organized as follows. We will not give the full proof of Theorem~1 here as the argument is described in the extended abstract~\cite{morters_et_al:LIPIcs.AofA.2024.14}. We will however give a completely self-contained proof of Theorem~2 and therefore include most arguments that are needed for the proof of Theorem~1. This proof  of Theorem~2 will be given in Section~2. Note that Theorem~2 also establishes the lower bound in Theorem~3 and
in Section~3 we complete the proof of Theorem~3 by providing an upper bound.

\section{Proof of Theorem 2}
\label{sec2.2}

For the proof of Theorem~2 we embed a Galton-Watson tree into our graph. To explain the idea fix  small parameters $0<u,b <1 $. Let $m=u^{k-1}n$ for some positive integer~$k$ (in the following, to avoid cluttering notation, we do not make the rounding of $m$ to an integer explicit). We explore the neighbourhood of a vertex $o_n$ with $bum \leq o_n\leq um$ in the graph $\mathscr G_{m}$. We will see below that this exploration can be coupled to a branching random walk killed upon leaving a bounded interval such that with high probability the number of particles near the right interval boundary exceeds the number $S_m(o_n)$ of vertices in  the connected component of $o_n\in\mathscr G_m$ with index at least $bm$. These vertices will be the offspring of the vertex $o_n$ in our Galton-Watson tree. Before describing this coupling  in detail we give a lower bound on the number of particles in the killed branching random walk. This result, formulated as Proposition~\ref{twokilled}, may be of independent~interest.\medskip

We denote by $\mathscr V$ the tree of Ulam-Harris labels, i.e.\ all finite sequences of natural numbers
including the empty sequence $\varnothing$, which denotes the root. Given a label $v=(v_1,\ldots, v_n)\in\mathscr V\setminus\{\varnothing\}$ we denote 
by $|v|=n$ its length, corresponding to the generation of vertex $v$ in the tree and by $\cev{v}=(v_1,\ldots, v_{n-1})$ the parent of $v$ in the tree. We attach to every label $v\in\mathscr V$ an independent sample $P_v$ of a point process with infinitely many points 
$P_v(1)\leq P_v(2) \leq P_v(3) \leq\ldots$ in increasing order  on the real line, in our case the Poisson process with intensity measure
$$\pi(dx)=\beta (e^{\gamma x} \mathbbm 1_{x>0}
+ e^{(1-\gamma) x} \mathbbm 1_{x<0}) \, dx.$$
We denote by $$\mathscr T(x)=
( V(v) \colon v \in\mathscr V)$$
the \emph{branching random walk} started in $x\in\mathbb R$, which is characterised by the position
$$V(v)= x + \sum_{i=1}^{|v|} P_{(v_1,\ldots, v_{i-1})}(v_i)$$
of the particle with label $v\in\mathscr V$. 
When started in $\log u$ we denote the underlying probability and expectation by $\mathbb P_u, \mathbb E_u$ and denote the branching random walk by $\mathscr T$. The Laplace transform of the branching random walk is given by
\begin{align*}
    \psi(t)= \mathbb{E}_1\Big[\sum_{\abs{v}=1}e^{-tV(v)}\Big] = \frac{\beta}{t-\gamma}+\frac{\beta}{1-\gamma-t} \quad\text{ if $\gamma<t<1-\gamma$,}
\end{align*}
and $\psi(t)=\infty$ otherwise.
The domain of $\psi$ is nonempty if $\gamma<\frac12$ and there exists $t>0$ with $\psi(t)<1$ if and only if~$0<\beta <\frac{1}{4}-\frac{\gamma}{2}$, i.e.\ in the subcritical regime for the inhomogeneous random graph. Under this assumption 
there exist $\rho_- < \rho_+$ with 
$\psi(\rho_-)=\psi(\rho_+)=1$.
We can calculate both values explicitly, 
$$
\rho_{\pm} = \tfrac{1}{2} \pm \sqrt{(\gamma-\tfrac{1}{2})^2+\beta(2\gamma-1)}.
$$
For $0\leq a<d$ denote by $\mathscr T_{a,d}(x)$ the \emph{killed branching random walk} starting with a particle in location~$x$, where all particles located outside the interval  $(\log a, \log d]$ are killed together with their descendants. Again we omit the starting point from the notation if it is clear from the context. Note that $v=(v_1,\ldots,v_n)\in\mathscr T_{a,d}$ means that, for all $0\leq i \leq n$,
$$ \log a < V(v_1,\ldots,v_i) \leq \log d.$$
Of particular interest is $\mathscr T_{0,1}$ where particles with positions on the positive half-line are killed. 
The condition $\gamma <\frac{1}{2}$ and $\beta <\frac{1}{4}-\frac{\gamma}{2}$ is necessary and sufficient for $\mathscr T_{0,1}(x)$ started at $x\leq0$ to suffer extinction in finite time almost surely, see~\cite[Theorem~1.3]{shi}.
\pagebreak[3]\medskip

For $0\leq a \leq b<1$ denote by  $I(a,b)$ be the total number of surviving particles of $\mathscr T_{a,1}$ located in $(\log b,0]$. 
{We prove a limit theorem 
for $I(0,b)$ under $\mathbb P_u$ when $u \downarrow0$.}

\begin{proposition}\label{twokilled}
{For every fixed $0\leq b<1$} the random variable {$I(0,b)$ satisfies
$$\lim_{u \downarrow0} \mathbb{P}_u \left( I(0,b) \geq xu^{-\rho_-}  \right) 
= \mathbb{P} \left( Y\geq x  \right),$$}
where $Y$ is a positive random variable  satisfying 
\begin{equation}\label{taily}
\mathbb{P}\left( Y \geq x \right) = x^{-(\rho_+/\rho_-) + o(1)} \text{ as $x\to\infty$.}
\end{equation}
\end{proposition}

Proposition~\ref{twokilled}  will be proved in Section~\ref{sec_twokilled}.
\medskip

This proposition will play a crucial role when we construct the simultaneous coupling of the neighbourhoods of vertices in $\mathscr G_m$. We use the projection
\begin{align} \label{pim}
    \pi_m \colon & (-\infty,0] \to \{1,\ldots,m\},
\end{align}
defined by
$$-\sum_{k=\pi_m(x)}^m \frac1k < x \leq -\sum_{k=\pi_m(x)+1}^m \frac1k$$
to map locations on the negative half-line to vertex numbers in $\mathscr G_m$. Its partial inverse is 
\begin{align*}
 \phi_m \colon & \{1, \ldots, m\} \rightarrow (-\infty,0] \quad , \quad  i \mapsto - \sum_{j=i+1}^{m} \frac{1}{j} \; .
\end{align*}
For any set $\mathcal U\subset \{1,\ldots, m\}$ we denote by $\mathscr F_{\mathcal U}$ the $\sigma$-algebra generated by the restriction of the random graph $\mathscr G_m$ to the vertex set $\mathcal  U$.
Let $\gamma<\rho<\rho_-$.
\medskip

\begin{proposition}\label{mainlemma}
 For every $0<b< 1$ there exist $\varepsilon >0$, $a>1$ and $0<u_0<b$ with the property that for every $0<u<u_0$ there exists $m(u)$ such that, for all $m\ge m(u)$ and any set $\mathcal U'\subset \{1,\ldots, um\}$ with $|\mathcal  U'|\leq a m^{\rho}$ and family of $d\leq m^{\rho}$ vertices in 
 $\mathcal  U'$ with 
$$bum<u_1 < \cdots < u_d \leq um, $$  
there exist\\[-5mm]
\begin{itemize}
\item a set  $\mathcal U'\subset\mathcal U\subset \{1,\ldots, m\}$ with $|\mathcal  U|\leq a(m/u)^{\rho}$,\\[-7mm]
\item conditionally given $\mathscr F_{\mathcal U'}$ independent random variables
 $X_1, X_2 , \ldots , X_{d}$  with
 \[
  X_i = 
  \begin{cases}
    \lceil\varepsilon u^{-\rho}\rceil, & \text{with probability } \varepsilon >0 ,\\
    0, & \text{with probability  } 1-\varepsilon,\\
  \end{cases}  
\]\\[-7mm]
\item subsets $\mathcal X_1, \ldots, \mathcal X_d\subset \mathcal U \cap \{bm,\ldots,m\}$ with $|\mathcal X_i|=X_i$ such that $\mathcal X_i$ is contained in the connected component  of  $u_i$ in $\mathcal U$.
\end{itemize}
\end{proposition}

Proposition~\ref{mainlemma} will be proved in Section~\ref{sec_mainlemma} using Proposition~\ref{twokilled}.
\bigskip

We now complete the proof of Theorem~\ref{2} using Proposition~\ref{mainlemma}. 
Take $o_n\in \mathscr G_n$ so that $$ {o_n\to\infty} \text{ and } { \frac{o_n}{n} \to 0}.$$
We fix $\delta>0$, $b=\frac12$, then $\varepsilon>0$ from Proposition~\ref{mainlemma} and $0<u<u_0$ so that $\frac{2\log \varepsilon}{\log u}<\frac{\delta}2$ and also that \smash{$\varepsilon^2>u^{\rho}$}. Let
$$k=\frac{\log (o_n/n)}{\log u}-1.$$
Then $o_n=u^{k+1}n$ and we set $m:=u^{k}n$. Take $n$ large enough such that $m \geq m(u)$ as defined in Proposition~\ref{mainlemma}. This is possible since $m=o_n/u \rightarrow \infty$ as $n\to\infty$.
\medskip\pagebreak[3]

In the first step we use Proposition~\ref{mainlemma} with  $d=1$ and $u_1=o_n$. We obtain $X_1$ vertices with index $\ge bm$ in the component  $S_m(o_n)$. These vertices constitute the children of the root and therefore the first generation of the embedded Galton Watson tree. Their indices lie in the interval
$(bu^{k}n, u^{k}n]$. In the second step we take these vertices and the set $\mathcal  U$ from the first step as input into 
Proposition~\ref{mainlemma} which we now use with a new, larger $m:=u^{k-1}n$, see Figure~\ref{fig1} for an illustration. Note that $d\leq m^{\rho}$ and the conditions of Proposition~\ref{mainlemma}  are satisfied so that we get a second  generation consisting
of disjoint subsets 
$\mathcal X_1, \ldots \mathcal X_d$ of the connected component of $o_n$ 
in~$\mathscr G_m$. These are the offspring of the $d$ children of the root. 
We continue this procedure for altogether $k$ steps until, in the last step, we reach $m=n$. The number of vertices thus created in the component of $o_n\in\mathscr  G_n$ is the total size of the first $k$ generations of a Galton-Watson tree with offspring variable $X_i$.
\medskip

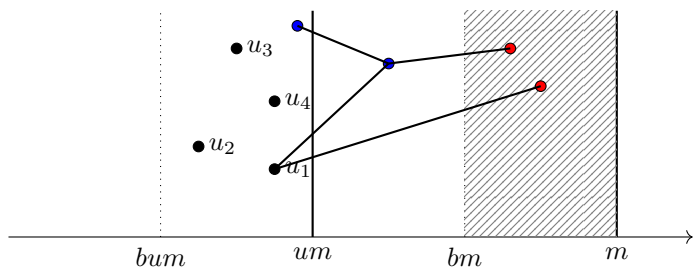
\begin{figure}
    \centering
    \begin{tikzpicture}
        \draw[->] (-6,0) -- (3,0) node[below] {};
        
        \draw[dotted] (-4,0) -- (-4,3);
        \node[below] at (-4,0) {\small $bum$};
        
        \draw[thick] (-2,0) -- (-2,3);
        \node[below] at (-2,0) {\small $um$};
        
        \draw[dotted] (0,0) -- (0,3);
        \node[below] at (0,0) {\small $bm$};
        
        \draw[thick] (2,0) -- (2,3);
        \node[below] at (2,0) {\small $m$};
        
        \fill[pattern=north east lines, pattern color=gray] (0,0) rectangle (2,3);
        
        \foreach \i/\y in {2/1.2, 3/2.5, 4/1.8} {
            \draw[fill=black] ({-4.5 + \i*0.5},{\y}) circle (2pt);
            \node[right] at ({-4.5 + \i*0.5},{\y}) {\small $u_{\i}$};
        }
        
        \draw[fill=black] (-2.5,0.9) circle (2pt);
        \node[right] at (-2.5,0.9) {\small $u_1$};
        
        \draw[fill=red] (1,2) circle (2pt); 
        \draw[fill=red] (0.6,2.5) circle (2pt);
        \draw[fill=blue] (-1,2.3) circle (2pt); 
        \draw[fill=blue] (-2.2,2.8) circle (2pt); 
        
        \draw[thick] (-2.5,0.9) -- (1,2); 
        \draw[thick] (-2.5,0.9) -- (-1,2.3); 
        \draw[thick] (-1,2.3) -- (-2.2,2.8); 
        \draw[thick] (-1,2.3) -- (0.6,2.5) ;
        
    \end{tikzpicture}
    \caption{Illustration of Proposition~\ref{mainlemma}: The vertices $u_1, \ldots, u_4$ are successively explored, the exploration of $u_1$ is depicted. The exploration yields particles in the entire interval $[bum,m]$ but only the red particles located in $[bm,m]$ are included in~$\mathcal X_1$. A logarithmic scale is used on the abscissa.\label{fig1}}
\end{figure}

As the mean offspring number is
$$\mathbb E[X_i] =\varepsilon \lceil \varepsilon u^{-\rho}\rceil >1,$$
the Galton-Watson tree is supercritical and survives forever with positive probability.~As 
$$k \sim - \frac{\log n}{\log u} \to\infty ,$$
on survival 
the number of vertices in the $k$th generation is a positive multiple of
\begin{align*}
    (\varepsilon \lceil \varepsilon u^{-\rho}\rceil)^k&= \exp{ -(1+o(1)) \frac{\log n}{\log u}(2\log \varepsilon  -\rho \log u)} \\
    &= \exp{ -(1+o(1)) \log n \Big(\frac{2\log \varepsilon}{\log u}   - \rho \Big)} 
    \geq n^{\rho-\delta},
\end{align*}
for all large $n$. In particular we have
$$S_n(o_n)\ge cn^{\rho-\delta}
 \quad \text{ for all $n$ with \emph{positive} probability.}$$ 
To get the result with \emph{high} probability we need to modify the first step of the construction
and start the Galton-Watson tree not with one but with a large but fixed number $d$ of vertices. We fix $0<b<1$ to be determined later and now let
$$k=\frac{\log (o_n/bn)}{\log u}-1$$
and note that $o_n=bum$ when we again set 
$m:=u^{k}n$. The difference between the degree of $o_n$ at times $um$ and $bum$ is the sum of $(1-b)um$
independent Bernoulli random variables with parameter
bounded from below by $\beta (um)^{\gamma-1}(bum)^{-\gamma}$. As $n\to\infty$ this random variable converges to a Poisson random variable with parameter
$\beta(1-b)b^{-\gamma}$. We can therefore make the probability that this random variable is larger 
than~$d$ arbitrarily close to one by picking a sufficiently small $b$ in our applications of Proposition~\ref{mainlemma}. On this event we can now start the construction with $d$ vertices which are all children of the original~$o_n$ and get $d$ independent supercritical Galton-Watson trees with the given offspring distribution. Observe that the survival of \emph{one} of these $d$ trees suffices to get the lower bound on $S_n(o_n)$
and the complementary probability of extinction of \emph{all} trees can be made arbitrarily small by choice of $d$. This completes the proof.

\subsection{Proof of Proposition~\ref{twokilled}}\label{sec_twokilled}

The idea of the proof is to exploit that, as $\psi(\rho_-)=1$, the process given by $$W_n:=\sum_{\abs{v}=n}e^{-\rho_-V(v)}$$ is a martingale. Since $W_n$ is nonnegative it converges to some limit $W$, which we show to be strictly positive. We then look at this martingale from the point of view of a stopping line, as discussed in~\cite{Kyp}. Theorem~9 in ~\cite{Kyp}
implies convergence as $t\to\infty$  of
$(e^{-\rho_-t} Z_t')$ to $W$, where
\begin{equation}\label{ky}
Z_t':= \sum_{v \in\mathscr V}
\mathbbm{1}_{\{V(v)<t\}}
\sum_{y \colon \cev{y}=v}
e^{-\rho_-(V(y)-t)} \mathbbm{1}_{\{V(y)\geq t\}}.
\end{equation}
Observe that conditional on the $v$ with $V(v)<t$ the inner sums are independent with a distribution depending continuously on $V(v)-t$. 
A result of Nerman~\cite{nerman} therefore gives that 
the inner sum can be replaced by $\mathbbm{1}_{\{t-V(v)\leq \log b\}}$ and we still get convergence to a constant multiple of $W$.
\bigskip

We start the detailed proof by verifying  that the limiting $W$ is strictly positive and satisfies the tail property of \eqref{taily}. By Biggins' theorem for branching random walks, see e.g.~\cite{biggins, lyons}, the martingale limit $W$ is strictly positive if and only if the following two conditions hold, 
\begin{itemize}
    \item[]$(i)\ \psi(\rho_-)-\frac{\rho_-\psi'(\rho_-)}{\psi(\rho_-)}>0 \, , $ \smallskip
    \item[]$(ii)\ \mathbb{E}_1[W_1 \log W_1]<\infty.$\smallskip
\end{itemize}
The first one holds as $\psi(\rho_-)=1$ and $\psi'(\rho_-)<0$. For the second condition it suffices to prove the following lemma.

\begin{lemma}\label{l1}
For $1<p<\frac{1-\gamma}{\rho_-}$ we have
$\mathbb E_1 \big[W_1^p\big]
<\infty.$
\end{lemma}

\begin{proof}
We define 
$$f(x,\Pi)= e^{-\rho_- V(x)} (\sum_{y\in\Pi}e^{-\rho_- V(y)})^{p-1}\, .$$ 
Then $\mathbb{E}_1[W_1^p]=\mathbb{E}[\int f(x,\Pi) \, \Pi(dx)]$ and by Mecke's equation~\cite[Theorem 4.1]{PPP} we get
\begin{align*}
    \mathbb{E}_1[W_1^p]&= \int \mathbb{E}[f(x,\Pi+\delta_x)] \, \pi(dx) = \int e^{-\rho_-x}\mathbb{E}\Big[\big(e^{-\rho_-x}+\int e^{-\rho_- t} \, \Pi(dt) \big)^{p-1}\Big] \pi(dx)\\
    &\leq 2^{p-1} \Big(\int e^{-p \rho_-x} \pi(dx) + \mathbb{E}_1\big[ W_1^{p-1}\big] \, \psi(\rho_-) \Big)\, .
\end{align*}
The left summand is equal to $\psi(p\rho_-)$ which is finite for $1<p<\frac{1-\gamma}{\rho_-}$.
The right summand is finite if $1<p\leq 2$ because in this case, by Jensen's inequality, the expectation is bounded by one. If $p>2$ we iterate the argument, using the same bound but now with \smash{$1<p-1<\frac{1-\gamma}{\rho_-}$}. In each iteration the exponent is reduced by one until it is no larger than two.
\end{proof}

Biggins' theorem ensures not only that $W>0$ on survival of $\mathscr T$ but also that $W_n\to W$ in $L^1$.  By the next lemma we can improve this to convergence in $L^p$ for
$p<\rho_+/\rho_-$.

\begin{lemma}\label{l2}
For $1<p<\rho_+/\rho_-$ we have that
$\displaystyle\sup_{n\in\mathbb N} \mathbb E_1 \big[W_n^p\big]
< \infty$ and $W_n\to W$ in $L^p$.
\end{lemma}

\begin{proof}
By 
Proposition~2.1 in \cite{Iksanov}
we get that $(W_n)$ converges in $L^p$ and that $\mathbb E_1[W_n^p]$ is bounded if
$$ \mathbb E_1 \big[W_1^p\big]
<\infty \text{ and } \psi(p\rho_-) < \psi(\rho_-)^p.$$
The first condition is verified under the weaker condition \smash{$1<p<\frac{1-\gamma}{\rho_-}$} in Lemma~\ref{l1}. As
$\psi(\rho_-)=1$ the second condition becomes
$\psi(p\rho_-) <1$, which holds if $p<\rho_+/\rho_-$.
\end{proof}

The tail behaviour of $W$ (and later of $Y$) claimed in Proposition~\ref{twokilled} now follows directly from
Lemma~\ref{l2} by Markov's inequality.
\bigskip

  As in
  \cite[Proposition~8]{morters_et_al:LIPIcs.AofA.2024.14} our next aim is to rewrite 
  $\mathscr T_{0,1/u}$ started at the origin 
  in terms of a sum over characteristics of the individuals in the population at time $t=-\log u$ of a  general (Crump-Mode-Jagers) branching process. In a general branching process the location of all offspring is to the right of the parent and locations are interpreted as birth-times of offspring particles.
  \medskip
    
   \begin{figure}[ht]
    \centering
\begin{tikzpicture}
    \draw[->] (-3,0) -- (4,0); 
    \draw[->] (0,-1) -- (0,4); 
    \draw[->] (0,-1) -- (4,-1);
    
    \foreach \x/\y in {-2/2, -1/1, 0/0} {
        \fill[blue!50!black] (\x,\y) circle (2pt);
    }
    \foreach \x/\y in {1/1, 2/1, 3/1, 0.3/2, 1.2/2, , 0.1/3, 1.7/3} {
        \fill[red] (\x,\y) circle (2pt);
    }
    \draw[blue!50!black, thick] (-2,2) -- (-1,1) -- (0,0);
    
    \foreach \x/\y in {-1.5/2, -0.8/3, -0.5/1} {
        \fill[blue!50!black] (\x,\y) circle (2pt);
    }
    \draw[blue!50!black, thick] (-1,1) -- (-1.5,2);
    \draw[blue!50!black, thick] (-2,2) -- (-0.8,3);
    \draw[blue!50!black, thick] (0,0) -- (-0.5,1);
   \draw[blue!50!black, thick] (0,0) -- (1,1);
   \draw[blue!50!black, thick] (0,0) -- (2,1);
   \draw[blue!50!black, thick] (0,0) -- (3,1);
    \draw[blue!50!black, thick] (-0.5,1) -- (0.3,2);

    \draw[blue!50!black, thick] (-1.5,2) -- (0.1,3);
    \draw[blue!50!black, thick] (-1.5,2) -- (1.7,3);

    \draw[blue!50!black, thick] (-0.5,1) -- (1.2,2);
    
    \foreach \x/\y in {1/1, 2/1, 3/1, 0.3/2, 1.2/2, , 0.1/3, 1.7/3} {
        \fill[red] (\x,-1) circle (2pt);
        \draw[red, thin] (\x,-1) -- (\x,\y);
    }
    
    \node at (-0.3,-0.3) {$0$};
    \node at (4,-1.3) {$\xi$};
    
\end{tikzpicture}
\caption{Branching particles are  marked in blue. The positions on $[0,\infty)$ of the frozen particles, which are marked in red,  yield the point process~$\xi$.\label{f2}}
\end{figure}
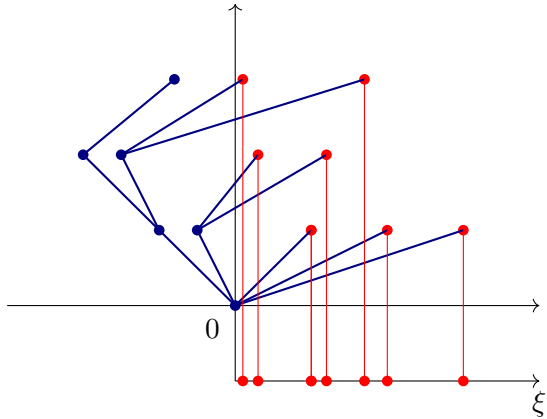

 To this end we divide the offspring of a particle $v=(v_1,\ldots, v_n)$ at location $V(v)$ into \emph{branching} particles to its left, and \emph{frozen} particles to its right.  
  The offspring of the branching particles is again divided into branching particles to the left of $V(v)$ and frozen particles to its right, until (after a finite number of steps) the offspring of all branching particles has been divided into branching
  and frozen particles. The frozen particles are all located to the right of $V(v)$, they constitute the offspring process of $v$ in the general branching process. 
  Their relative positions form a point process 
$$\xi_v=\sum_{{w\in \mathscr V, |w|>n}\atop{(w_1,\ldots, w_n) = (v_1,\ldots, v_n)}}
\delta_{V(w)-V(v)}
\mathbbm{1}_{\{V(w)> V(v) ,  V(w_1,\ldots,w_i)\leq  V(v) \forall n\leq i<|w|\}},$$ 
and they are all copies of the point process $\xi$ depicted in  Figure~\ref{f2}. The branching particles form a set $\mathscr B_v$
and their locations are all to the left of $V(v)$. \pagebreak[3]\medskip

To construct the general branching process we start with the root located at the origin,  considered initially to be frozen, take the point process $\xi_\varnothing$ of frozen particles as birth times of the children of the root and apply the same  procedure to every child $v$ of the root. The processes $\xi_v$ and cardinalities $|\mathscr B_v|$
are independent and identically distributed
over all the frozen particles $v$. The total number 
of particles in $\mathscr T_{0,1/u}$ equals
$$\sum_{{v \text{ frozen}}\atop{V(v) \leq t}} (1+ |\mathscr B_v|),$$
where $t=-\log u$.
To obtain convergence of this  quantity (properly scaled) we need to find the 
Malthusian parameter $\alpha>0$ associated to $\xi$, defined by
$$\mathbb E \int_0^\infty e^{-\alpha t} \xi(dt)=1.
$$
We now show that $\rho_-$ is the Malthusian parameter associated to $\xi$. To this end we construct a martingale $(M_n)$
as follows: We start with a particle at the origin and $M_0=1$. In every step 
we replace the leftmost particle 
by its offspring chosen with displacements according to a Poisson process of intensity~${\pi}$ and leave all other particles alive. Particles in $(0,\infty)$ never branch and remain alive but frozen. 
If the leftmost particle is in $(0,\infty)$ the process stops and 
the positions of the frozen particles make up~$\xi$. The random variable
$M_n$ is obtained as the sum of all particles~$x$ alive  after the $n$th step weighted with $e^{-\rho_- V(x)}$. Because $\psi(\rho_-)=1$ the process $(M_n)$ is indeed a martingale, and it clearly converges almost surely to $$M_\infty=\int_0^\infty \e^{-\rho_- t} \xi(\dd{t}).$$ Now take $\alpha>\rho_-$ with $\psi(\alpha)<1$. Then $M_n$ is dominated by
\begin{align*}
    M_n \leq \sum_{u \text{ branching}} e^{- \alpha V(u)} + \sum_{u \text{ frozen}}e^{-\rho_-V(u)}
\end{align*}
 The right-hand side is integrable, as the sum over frozen particles born from a single particle $x$ in position $V(x)<0$ has expectation at most $e^{-\alpha V(x)}$ and the expected sum over these bounds for all branching particles is itself bounded by \smash{$\frac1{1-\psi(\alpha)}$}. By dominated convergence, we get that 
 $\mathbb E[M_\infty]=1$ and hence
 $\rho_-$ is the Malthusian parameter.
\medskip\pagebreak[3]

Theorem~3.1 in Nerman~\cite{nerman} yields 
convergence of $(e^{-\rho_- t} Z_t^\phi)$ to 
a positive random variable $m_\phi Z$ for
\begin{align*}
        Z_t^\phi:=\sum_{v\colon V(v) \leq t}\phi_v(t-V(v)),
    \end{align*}
where the sum is over the particles of the general branching process born before time~$t$
and the characteristics $\phi_v$  are 
independent, identically distributed copies of a random function  $\phi\colon [0,\infty) \to [0,\infty)$ satisfying mild technical conditions. Moreover, $Z$ is a positive random variable independent of $\phi$ and $m_\phi$ a positive constant depending on~$\phi$. The conditions of~\cite{nerman} are satisfied for the process $(Z_t')$ in~\eqref{ky} by \cite[Corollary 2.5]{nerman}, whence $Z$ is a constant multiple of $W$, but also when the processes $(\phi(s) \colon s\ge 0)$ are  
bounded  by an integrable random variable and
$\mathbb E\phi \colon [0,\infty) \to [0,\infty)$ is continuous. 
\medskip

We now look at the total number $I(0,b)$ of surviving particles of $\mathscr T_{0,1}(\log u)$ located in $(\log b,0]$.
 We shift all particle positions by $t = -\log u$. Then the killed branching random walk $ \mathscr T_{0,1}(\log u)$ becomes a killed branching random walk $\mathscr T_{0,1/u}(0)$ and $I(0,b)$ the number of surviving particles in $(t+\log b,t]$. \medskip
 
 We have $I(0,b)=Z_t^\phi$ for the general branching process with offspring law $\xi$ at the time $t = -\log u$ and for the characteristic 
 $$\phi_v(s)= \sum_{w \in \mathscr B_v}
 \mathbbm{1}_{\{s+\log b< V_v(w) \leq s\}},$$ 
 where $V_v(w)$ is the relative position of the branching particle~$w$ to $v$. Then $\phi_v(t-V(v))$ is the number of branching particles descending from $v$ (including $v$ itself) located in the interval
 $(t+\log b,t]$. This process is dominated by
 $1+|\mathscr B_v|$. To check that $|\mathscr B_\varnothing|$ is integrable, fix $\alpha>0$ with $\psi(\alpha)<1$. Then we have for $v\in\mathscr B_\varnothing $ that $e^{-\alpha V(v)} \geq 1$ and $$\mathbb{E} \sum_{{v\in\mathscr B_\varnothing}\atop{\abs{v}=n}} e^{-\alpha V(v)} \leq \psi(\alpha)^n.$$ Hence
$\mathbb{E}[|\mathscr B_v|]\leq  
\sum_n \psi(\alpha)^n
\leq \frac{1}{1-\psi(\alpha)} < \infty.$ As
$\mathbb{E} \phi_v$ is clearly continuous the conditions of~\cite[Theorem~3.1]{nerman} on the characteristics are satisfied.
 \medskip   
 
Altogether this yields that 
$$\lim_{u\downarrow 0} u^{\rho_-}I(0,b)= 
\lim_{t\uparrow \infty}
e^{-\rho_- t} Z_t^\phi = Y \quad\text{ in distribution,}$$ 
where  the limit $Y$ is a positive, constant multiple of the positive martingale limit $W$. 
    \smallskip

\subsection{Proof of Proposition~\ref{mainlemma}}\label{sec_mainlemma}

Under the assumption~$0<\beta <\frac{1}{4}-\frac{\gamma}{2}$ the leftmost particle of $\mathscr T$ drifts to the right, i.e.\ 
$\lim_{n\to\infty} \inf_{|v|=n} V(v)=\infty $, see~\cite[Lemma~3.1]{shi}. Hence
$\inf_{v\in\mathscr V} V(v)$ is a finite random variable
and it is easy to see that in our case its support is the entire negative half-line. Hence,
given $0<b<1$, we can pick $\varepsilon>0$ such that
$$\mathbb P_1\big( \inf_{v\in\mathscr V} V(v) > \log b\big)  \geq \varepsilon.$$
Additionally we request that, for $Y$ as in Proposition~\ref{twokilled}, $\varepsilon>0$ satisfies
$\mathbb{P} \left( Y\geq \varepsilon  \right)
\geq 5\varepsilon.$\pagebreak[3]
This implies that
$$\liminf_{u \downarrow0} \inf_{u'\in[ub,u]} \mathbb{P}_{u'} \left( I(ub,b) \geq \varepsilon u^{-\rho_-}  \right) 
\geq  \liminf_{u \downarrow0} \mathbb{P}_u \left( I(0,b) \geq \varepsilon u^{-\rho_-}  \right) - \varepsilon \geq
4 \varepsilon$$
and, for suitably large $a>1$,
$$\limsup_{u \downarrow0} \sup_{u'\in[ub,u]} \mathbb{P}_{u'} ( I(ub,0) \geq 
a u^{-\rho_-}) \leq \limsup_{u \downarrow0} \mathbb{P}_{ub} ( I(0,0) \geq 
a u^{-\rho_-}) \leq \tfrac{\varepsilon}2.$$
We pick $0<u_0<b \wedge 2^{-1/\rho_-}$  
such that $\inf_{u'\in[ub,u]} \mathbb{P}_{u'} ( I(ub,b) \geq \varepsilon u^{-\rho_-}) 
\geq 3\varepsilon$ and also $a>1$  such that
$\sup_{u'\in[ub,u]} \mathbb{P}_{u'} ( I(ub,0) \geq 
\frac{a}2 u^{-\rho_-}) \leq \varepsilon$, for all $0<u<u_0$. The exploration algorithm below uses $\varepsilon, a, \rho_-$ and $u_0$ 
as derived above from the parameter $\pi$.%
\bigskip%

\noindent
We present, for parameters $(\pi, u, m)$ 
with $m\in\mathbb N$, 
an exploration algorithm with~input
\begin{itemize}
    \item a graph $\mathscr U'\subset \{1,\ldots, um\}$ with at most $a m^{\rho_-}$ vertices,
     \item distinct vertices  $u_1 < \ldots < u_d$ in $\mathscr U'$ with $bum<u_i\leq um$ and $d\leq m^{\rho_-}$.
\end{itemize}
The output of the algorithm are
\begin{itemize}
    \item a family of pairwise disjoint sets $\mathcal Y_1, \ldots, \mathcal Y_d \subset\{bm,\ldots,m\}$,
    \item a graph $\mathscr U \subset \{1,\ldots, m\}$ with at most $a (\frac{m}u)^{\rho_-}$ vertices such that 
    $\mathscr U'$ is an embedded subgraph and the sets $\mathcal Y_i$ are contained in the connected component of $u_i$ in~$\mathscr U$.
\end{itemize}

\begin{algorithm}[h]
\caption{Branching Random Walk Exploration~$(\pi,u,m)$}
\begin{algorithmic}[1]
\State Set $\mathscr U :=\mathscr U'$ and $i:=1$.
\While{$i \leq d$}
    \State Set $\mathcal B_i :=\emptyset$.
    \State Sample a killed branching random walk $\mathscr T_{ub,1}$ with intensity ${\pi}$ and start in~$\phi_m(u_i)$.
    \For{all particles $\varnothing\not=v\in\mathscr T_{ub,1}$  in a depth-first exploration}
        \If{$\pi_m(V(v)) \in \mathscr{U}$} \label{error 1}
            \State Output $\mathcal{Y}_i = \emptyset$, set $i:=i+1$ and \textbf{end while}
        \EndIf
        \State Add vertex $\pi_m(V(v))$ and the explored edge leading to it to the graph $\mathscr{U}$
        \State Add $\pi_m(V(v))$ to the set $\mathcal{B}_i$
\If{$|\mathcal{B}_i|\ge \frac{a}2 u^{-\rho_-}$} \label{error 2}
            \State Output $\mathcal{Y}_i = \emptyset$, set $i:=i+1$ and \textbf{end while}
        \EndIf        
        \If{$V(v) \in [\log b, 0]$}
            \State Add $\pi_m(V(v))$ to $\mathcal{Y}_i$
        \EndIf
        \If{$|\mathcal{Y}_i| \geq \varepsilon u^{-\rho_-}$}
            \State Output $\mathcal{Y}_i$, set $i:=i+1$ and \textbf{end while}
        \EndIf
    \EndFor
        \State Output $\mathcal{Y}_i = \emptyset$, set $i:=i+1$. \label{error 3}
        \EndWhile
\State Output $\mathscr U$.
\end{algorithmic}
\end{algorithm}

By construction the output sets 
$\mathcal Y_1, \ldots, \mathcal Y_d$ are pairwise disjoint and $u_i$ is connected to~$\mathcal Y_i$
by edges in $\mathscr U$. Also, for every $i\in\{1,\ldots,d\}$ the algorithm adds at most $\frac{a}2u^{-\rho_-}+1$ vertices to the graph~$\mathscr U$, so that its output $\mathscr U$ satisfies
\begin{align*}
|\mathscr U| & \leq |\mathscr U'|+
d\big(\tfrac{a}2u^{-\rho_-}+1\big) \leq 
a (m/u)^{\rho_-} \big( u^{\rho_-} +\tfrac12 + \tfrac1a u^{\rho_-}\big) \leq a (m/u)^{\rho_-},
\end{align*}
for all $0<u<u_0$ by choice of $u_0$. In the following we show how the algorithm can be used to construct a suitably large subgraph of $\mathscr G_m$.
\bigskip


We run the algorithm with parameter $(\tilde\pi,u,m)$ for an intensity measure with a slightly decreased
density parameter $0<\tilde\beta<\beta$, $0<u<u_0$ and some large~$m$. This leads to a slightly smaller value of $\rho_-$ which is referred to as $\rho$ in the statement of Proposition~\ref{mainlemma}. The next lemma shows that the probability of edges inserted by the algorithm is bounded from above by the edge probabilities in~$\mathscr G_m$.

\begin{lemma}\label{coupling}
There exists $m(u)\in\mathbb N$ such that, for all  $m\ge m(u)$, for all $m \ge s,r \geq bum$ with $s \not=r$ the probability that a particle~$v$ in location $V(v)$ with $\pi_m(V(v))=r$ has an offspring~$y$ with location $V(y)$ satisfying $\pi_m(V(y))=s$
is at most
$$\beta (r\wedge s)^{-\gamma} (r\vee s)^{\gamma-1}.$$
\end{lemma}

\begin{proof} 
For a fixed particle~$v$ in location $V(v)$ with $\pi_m(V(v))=r$ the probability that it  has an offspring~$y$ with location $V(y)$ satisfying $\pi_m(V(y))=s$ equals
\begin{equation}\label{novert1}
1-\exp\bigg(-\tilde\pi\Big( -\sum_{k=s}^m \frac1k -V(v), -\sum_{k=s+1}^m \frac1k-V(v)\Big]\bigg).
\end{equation}
As $\pi_m(V(v))=r$ we have
$$-\sum_{k=r}^m \frac1k < V(v) \leq -\sum_{k=r+1}^m \frac1k.$$
The probability in \eqref{novert1} is therefore largest when $V(v)=-\sum_{k=r}^m \frac1k$. It therefore remains to show that, for $bum\leq s<r$, we have
\begin{equation}\label{a01}
1-\exp\bigg(-\tilde\pi\Big( -\sum_{k=s}^{r-1} \frac1k, -\sum_{k=s+1}^{r-1}\frac1k\Big]\bigg)
\leq \beta s^{-\gamma}r^{\gamma-1},
\end{equation}
and, for $bum\leq r<s$, we have
\begin{equation}\label{a02}
1-\exp\bigg(-\tilde\pi\Big(\sum_{k=r}^{s-1} \frac1k, \sum_{k=r}^{s} \frac1k\Big]\bigg)
\leq \beta s^{\gamma-1}r^{-\gamma}.
\end{equation}
For \eqref{a01} we find that,
for some constant $C>0$, if $m\geq m(u)$ for a suitable $m(u)\in\mathbb N$,
\begin{align*}
    \tilde\pi\Big( -\sum_{k=s}^{r-1} \frac1k, -\sum_{k=s+1}^{r-1}\frac1k\Big] & = \frac{\tilde\beta}{1-\gamma}\, \exp\Big({-(1-\gamma)\sum_{k=s}^{r-1} \frac1k} \Big)(e^{\frac{1-\gamma}{s+1}}-1)\\
    & \leq  \Big(\frac{\tilde\beta}{s+1}+ \frac{C}{(bum)^2}\Big)\, \exp\big(-(1-\gamma)
    (\log(\tfrac{r-1}{s-1}) - \tfrac{C}{bum}) \big)\\[2mm]
    & \leq  \beta  s^{-\gamma}r^{\gamma-1}.
\end{align*}
Hence, using that $1-e^{-x}\le x$, we get \eqref{a01}.
\pagebreak[3]\medskip

For \eqref{a02} we find that,
for some constant $C>0$, if $m\geq m(u)$ for a suitable $m(u)\in\mathbb N$,
\begin{align*}
    \tilde\pi\Big( \sum_{k=r}^{s-1} \frac1k, \sum_{k=r}^{s}\frac1k\Big] & = \frac{\tilde\beta}{\gamma}\exp\Big({\gamma\sum_{k=r}^{s-1} \frac1k} \Big)(e^{\frac{\gamma}{s}-1})\\
    & \leq  \Big(\frac{\tilde\beta}{s}+\frac{C}{(bum)^2}\Big)\, \exp\Big(\gamma
    (\log(\tfrac{s-1}{r-1}) - \tfrac{C}{bum}) \Big)\\[2mm]
    & \leq  \beta  s^{\gamma-1}r^{-\gamma}.
\end{align*}
Hence, using that $1-e^{-x}\le x$, we get \eqref{a02}.
\end{proof}

Let $E_i$ be the event that the exploration of $u_i$ was successful. This is the case if $\mathcal Y_i\not=\emptyset$ or, equivalently, $\abs{\mathcal Y_i}\geq \varepsilon u^{-\rho_-}$. Let
$\mathscr U_i$ be the graph in Algorithm~1 at the time
when the exploration of $u_i$ is completed and
$(\mathscr F_{\mathscr U_i} \colon i=0,\ldots,d)$ the natural filtration associated with this process. Note that 
$$\mathscr U'=:\mathscr U_0 \subset \mathscr U_1 \subset \ldots \subset \mathscr U_d=\mathscr U,$$
and that $E_i\in\mathscr F_{\mathscr U_i}$ for all $i\in\{1,\ldots,d\}$.
\bigskip

\begin{lemma}\label{successlemma}
For $0<u <u_0$ there exists $m(u)\in\mathbb N$ such that, for all $m \geq m(u)$, almost surely,
$$\mathbb P(E_{i+1} \mid \mathscr F_{\mathscr U_i}) \ge \varepsilon.$$ 
\end{lemma}

\begin{proof}
Let $i \in \{0,\dots,d-1\}$. Working on $\mathscr{F}_{\mathscr U_i}$ we know the graph 
$\mathscr{U}_{i}$ and the algorithm explores the branching random walk $\mathscr{T}_{ub,1}(\phi_m(u_{i+1}))$. We have to control the probability that the algorithm stops without $E_i$. This can happen on three different occasions: 
\begin{itemize}
\item \emph{Line~\ref{error 1}:} For an explored particle $v$ we have that $\pi_m(V(v)) \in \mathscr{U}_i$. \smallskip

Since $\pi_m(V(v)) \in (bum,m]$ we can use Lemma~\ref{coupling}, and find $m(u)\in \mathbb N$ such that for all $m \geq m(u)$ we can upper bound the probability that $V(v)$ is in a region that gets projected to a fixed vertex $j \in \mathscr{U}_i$  by
$$  \beta \left( (\pi_m(V(v)) \wedge j)^{-\gamma} (\pi_m(V(v)) \vee j)^{\gamma-1}  \right) \leq \frac{\beta}{bum}.$$
There are at most $\abs{\mathscr{U}_i}\leq a u^{-\rho_-}m^{\rho_-}$ such vertices. Therefore we get
\begin{align*}
    \mathbb{P}( \pi_m(V(v)) \in \mathscr{U}_i) \leq \frac{\beta |\mathscr{U}_i|}{bum} = a\beta b^{-1} u^{-\rho_--1}m^{\rho_--1} \, .
\end{align*}
Due to the condition in line~\ref{error 2}, there are at most $\frac{a}{2}u^{-\rho_-}+1$ exploration steps where we have to account this error before the algorithm stops. Hence we can bound the probability in the complete exploration of the tree (the complete for-loop) by 
\begin{align*}
    \big(\tfrac{a}{2}u^{-\rho_-}+1\big)u^{-\rho_- -1} a \beta b^{-1} m^{\rho_--1}.
\end{align*}
Increase $m(u)$ if necessary so that for $m \geq m(u)$ this probability is bounded by~$\varepsilon$. 
\item \emph{Line~\ref{error 2}:} During the exploration we find more than $\frac{a}{2}u^{-\rho_-}$ vertices. \smallskip

By choice of $a$ and $u_0$ we have
$$\mathbb P(\abs{\mathcal{B}_i} \geq \tfrac{a}{2}u^{-\rho_-})=  \mathbb{P}_{u_{i+1}} ( I(ub,0) \geq \tfrac{a}2 u^{-\rho_-}) \leq \varepsilon \, .$$ 
\item \emph{Line~\ref{error 3}:} We do not find at least $\varepsilon u^{-\rho_-}$ vertices that we can add to $\mathcal{Y}_i$. \smallskip

This probability is bounded by $\mathbb{P}_{u_{i+1}} ( I(ub,b) < \varepsilon u^{-\rho_-}) 
\leq1- 3\varepsilon$. 
\end{itemize}
Taking a union bound we get $\mathbb P(E_{i+1}^{\mathrm c} \mid \mathscr F_{\mathscr U_i})\leq 1-\varepsilon$, as requested.
\end{proof}

To complete the construction we have to remove the possible dependence of the size of the sets
$\mathcal Y_1,  \ldots , \mathcal Y_{d}$ by means of the following decoupling lemma. 

\begin{lemma}\label{decoupling}
  Let   $\mathcal Y_1,  \ldots , \mathcal Y_{d}$  be random sets such that, almost surely,
  $$\mathbb P(|\mathcal Y_{i+1}|\ge k \mid \mathscr F_{\mathscr U_i}) \ge \epsilon,$$
  then there exist random sets ${\mathcal X_i} \subset \mathcal Y_i$ with $X_i$ elements such that
  $X_1,\ldots X_d$ are independent and 
  $\mathbb P(X_{i}= k) = \epsilon \text{ and }
 \mathbb P(X_{i}=0) = 1- \epsilon.$
\end{lemma}

\begin{proof}
    Let $U_1,\ldots, U_d$ be independent and uniformly distributed on $(0,1)$. Let $\mathscr E_1$ be the event that $\mathcal Y_1$ has at least $k$  elements and \smash{$U_1\leq \frac\epsilon{\mathbb P(|\mathcal Y_1|\ge k)}$} so that $\mathbb P(\mathscr E_1)=\epsilon$. On the event $\mathscr E_1$ we draw $k$ elements from $\mathcal Y_1$ without replacement and put ${\mathcal X_1}$ to be the set of elements thus drawn. On $\mathscr E_0=\mathscr E_1^c$
    we set ${\mathcal X_1}=\emptyset$. Then $X_1$ has the desired distribution.\smallskip

    Now let $\mathscr E_{i1}$ be the event $\mathscr E_i$ intersected with the event that $\mathcal Y_2$ has at least $k$  elements and $U_2\leq \frac\epsilon{\mathbb P(|\mathcal Y_2|\ge k \mid \mathscr E_i)}$. Then 
    $$\mathbb P(\mathscr E_{i1})=
    \mathbb P(\mathscr E_i) \mathbb P\big( |\mathcal Y_2|\ge k, U_2\leq \tfrac\epsilon{\mathbb P(|\mathcal Y_2|\ge k \mid \mathscr E_i)}\mid \mathscr E_i\big)
    = \epsilon \mathbb P(\mathscr E_i).$$ 
    On the event $\mathscr E_{i1}$ we draw $k$ elements from $\mathcal Y_2$ without replacement and put ${\mathcal X_2}$ to be the set of elements thus drawn. On $\mathscr E_{i0}=\mathscr E_i \setminus \mathscr E_{i1}$ we set ${\mathcal X_2}=\emptyset$. Then
    \begin{align*}
        \mathbb P(X_1=k, X_2=k) = \mathbb P( \mathscr E_{11}) = \epsilon^2, \qquad & \mathbb P(X_1=0, X_2=k) = \mathbb P( \mathscr E_{01}) = (1-\epsilon)\epsilon, \\
        \mathbb P(X_1=k, X_2=0) = \mathbb P( \mathscr E_{10}) = \epsilon(1-\epsilon), \qquad & \mathbb P(X_1=0, X_2=0) = \mathbb P( \mathscr E_{00}) = (1-\epsilon)^2. 
    \end{align*}
    This implies that $X_2$ has the desired distribution and $X_1$ and $X_2$ are independent.
    We continue with this method until $\mathcal X_1 , \ldots , \mathcal X_{d}$ are constructed.
\end{proof}

\begin{proof}[Proof of Proposition~\ref{mainlemma}]
To prove Proposition~\ref{mainlemma} we run Algorithm~1 with parameters $(\tilde\pi,u,m)$ for the intensity measure 
$$\tilde\pi(dx)=\tilde\beta (e^{\gamma x} \mathbbm 1_{x>0}
+ e^{(1-\gamma) x} \mathbbm 1_{x<0}) \, dx,$$ 
with a slightly decreased density parameter $0<\tilde\beta<\beta$, $0<u<u_0$ and~$m\ge m(u)$. Input are the vertices $u_1,\ldots, u_d\in \mathcal U'$ and a graph $\mathscr U'$ distributed like the restriction of~$\mathscr G_m$ to its vertex set $\mathcal U'$. Lemma~\ref{coupling} ensures that the algorithm inserts an edge into~$\mathscr U$ with a probability no larger than the edge probabilities in~$\mathscr G_m$. Hence the output graph~$\mathscr U$ is stochastically dominated by the restriction of~$\mathscr G_m$ to its vertex set, denoted~$\mathcal U$. 
\medskip%

By Lemma~\ref{successlemma} the set $\mathcal U$ contains disjoint subsets $\mathcal{Y}_1, \ldots, \mathcal{Y}_d$ with
$$\mathbb P(\abs{\mathcal Y_{i+1}}\geq \varepsilon u^{-\rho_-} \mid \mathscr F_{\mathscr U_i}) \ge \varepsilon,$$   and Lemma~\ref{decoupling} gives the existence of random sets $\mathcal X_i \subset \mathcal Y_i$ with
size $X_i=|\mathcal X_i|$ such that
    \[
    X_i =
    \begin{cases}
        \lceil \varepsilon u^{-\rho_-} \rceil, & \text{with probability } \varepsilon > 0, \\
        0, & \text{with probability } 1 - \varepsilon,
    \end{cases}
    \]
    and $X_1,\ldots, X_d$ independent.
By construction the connected component of $u_i$ in $\mathscr{U}$ contains  $\mathcal{X}_i$ with $\mathcal{X}_i \cap \mathcal{X}_j = \emptyset$ for all $i \neq j$. The construction can be completed by embedding 
$\mathscr{U}$ into~$\mathscr{G}_m$ by adding 
independent edges.\end{proof}


\section{Proof of Theorem~3.}

The lower bound follows directly from Theorem~2. We need to provide a matching upper bound.
Let $Z_k=| \{v \colon |\mathscr C_n(v)|\ge k\} |$ where  
$\mathscr C_n(v)$ is the connected component of $v$ in $\mathscr G_n$. Then we have
\begin{equation}\label{basic}
\mathbb P  \big(| S^{\text{max}}_n| \ge k\big) 
= \mathbb P  \big(Z_k \ge k\big) 
\leq \frac1k \mathbb E Z_k
= \frac{n}k \mathbb P  \big(| \mathscr C_n(O_n)| \ge k\big),
\end{equation}
where $O_n\in\{1,\ldots,n\}$ is uniformly chosen. We complete the argument in two steps. First, we dominate the graph $\mathscr C_n(O_n)$ by a branching random walk $\mathscr T_{0,1}(-X)$ started in a random placed vertex $-X$ and killed upon leaving the negative half-axis. For this purpose we slightly increase the edge density parameter $\beta$ in the definition of $\mathscr T_{0,1}(-X)$. The domination then holds unless the branching random walk visits a point located near (or to the left of) 
$-\log n$. The following proposition will be proved in Section~\ref{3.1}.

\begin{proposition}\label{coupling2}
For any $\beta<\tilde\beta<\beta_c$ and $\epsilon>0$ there is a coupling of $\mathscr G_n$ and a killed branching random walk $\mathscr T_{0,1}(-X)$ with intensity measure
$$\tilde\pi(dx)=\tilde\beta (e^{\gamma x} \mathbbm 1_{x>0}
+ e^{(1-\gamma) x} \mathbbm 1_{x<0}) \, dx$$
and standard exponential $X$ such that, for sufficiently large~$n$,
$$\mathbb P\big( | \mathscr C_n(O_n)| > |\mathscr T_{0,1}(-X)|\big)
\leq 
\mathbb P\big( \exists x \in \mathscr  T_{0,1}(-X) \text{ with } V(x) \leq  -(1-\epsilon)\log n \big).$$
\end{proposition}

The second step is to show that the probability that the killed branching random walk has a 
particle~$x$ with location $V(x) \leq  -(1-\epsilon)\log n$ or that its total progeny contains
substantially more than $n^{\rho_-}$ points is sufficiently small.

\begin{proposition}\label{ldp}
For all $\epsilon>0$ and sufficiently large~$n$ we have
\begin{itemize}
    \item[(a)] $\displaystyle\mathbb P\big( \exists x \in \mathscr  T_{0,1}(-X) \text{ with } V(x) \leq  -(1-\epsilon)\log n \big) \leq n^{-\rho_++\epsilon}.$
    \item[(b)] $\displaystyle \mathbb P  \big( |\mathscr T_{0,1}(-X)| \ge n^{\rho_-+\epsilon} \big) \leq n^{-\rho_++\epsilon}.$
\end{itemize}
\end{proposition}

For convenience we prove Proposition~\ref{ldp} in Section~\ref{3.2}.   for
the original intensity measure $\pi$, but of course the result can be applied to $\tilde\pi$ with a density parameter $\beta<\tilde\beta<\beta_c$ so close to $\beta$ that the difference of the resulting $\rho_\pm$ to the original value is as small as required. 
\medskip

Combining Propositions~\ref{coupling2} and~\ref{ldp} we get, for all $\epsilon>0$ and sufficiently large~$n$, that
\begin{align*}
\mathbb P  \big(|\mathscr C_n(O_n)| \ge n^{\rho_-+\epsilon} \big)
&\leq \mathbb P  \big(|\mathscr C_n(O_n)| \ge |\mathscr T_{0,1}(-X)| \big)
+ \mathbb P  \big( |\mathscr T_{0,1}(-X)| \ge n^{\rho_-+\epsilon} \big)\\
& \leq  2n^{-\rho_++\epsilon}.
\end{align*}
Finally combining this with \eqref{basic} we infer that
\begin{align*}
\mathbb P  \big(| S^{\text{max}}_n| \ge n^{\rho_-+2\epsilon}\big) &
\leq  n^{1-\rho_--2\epsilon} \mathbb P  \big( | \mathscr C_n(O_n)| \ge n^{\rho_-+\epsilon}\big) \\ &
 \leq 2n^{1-(\rho_-+\rho_+)-\epsilon}= 2n^{-\epsilon} \to0,
 \end{align*}
as required.\medskip

\subsection{Proof of Proposition~\ref{coupling2}}
\label{3.1}

For the coupling we sample a killed branching random walk 
$\mathscr T_{0,1}(-X)$
started in $-X$ with intensity measure $\tilde\pi$ which has a slightly increased (but still subcritical) density parameter $\tilde\beta>\beta$. All particles on the positive halfline and their descendants are killed. We use the projection $\pi_n$ defined in \eqref{pim}
to project all particle locations on the negative half-line onto vertices 
in~$\{1,\ldots,n\}$ and retain all edges as in the genealogical tree of $\mathscr T_{0,1}(-X)$. We denote the resulting multigraph by~$\mathscr K_n$.
To prove that this coupling has the properties claimed in  Proposition~\ref{coupling2} we use the following lemma.

\begin{lemma}\label{cou1}
   There exists $n_0\in\mathbb N$ such that, for all sufficiently large $n$, 
\begin{itemize}   
   \item[(i)] for all 
   $n \ge m,r \geq n_0$ with $m \not=r$ the probability that a particle~$x$ in location $V(x)$ with $\pi_n(V(x))=r$ has an offspring~$y$ with location $V(y)$ satisfying $\pi_n(V(y))=m$
is at least 
$$\beta (r\wedge m)^{-\gamma} (r\vee m)^{\gamma-1}.$$
\item[(ii)] for all $n \ge r > n_0$ the probability that a particle~$x$ in location $V(x)$ with $\pi_n(V(x))=r$ has at least one offspring $y$ with location $V(y)$ satisfying $\pi_n(V(y))\leq n_0$ is at least $$1-\prod_{m=1}^{n_0} (1-\beta m^{-\gamma}r^{\gamma-1}).$$ 
\end{itemize}
\end{lemma}

When the killed branching random walk is sampled we first check whether it has a particle~$y$ with location $V(y)$ satisfying $\pi_n(V(y))\leq n_0$. If this is the case then, for sufficiently large $n$, we have $V(y) \leq  -(1-\epsilon)\log n$ . If this is not the case, then all particle locations of $\mathscr T_{0,1}(-X)$ are projected  onto vertices with index at least
$n_0$. Using Lemma~\ref{cou1} to compare the edge probabilities, we see that in this case
$\mathscr G_n$ is dominated by $\mathscr K_n$, which proves Proposition~\ref{coupling2}.

\begin{proof}[Proof of Lemma~\ref{cou1}$(i)$]
The probability of the event that a fixed particle~$x$ in location $V(x)$ with $\pi_n(V(x))=r$ has an offspring~$y$ with location $V(y)$ satisfying $\pi_n(V(y))=m$ equals
\begin{equation}\label{novert}
1-\exp\bigg(-\tilde\pi\Big( -\sum_{k=m}^n \frac1k -V(x), -\sum_{k=m+1}^n \frac1k-V(x)\Big]\bigg).
\end{equation}
As $\pi_n(V(x))=r$ we have
$$-\sum_{k=r}^n \frac1k < V(x) \leq -\sum_{k=r+1}^n \frac1k.$$
The probability in \eqref{novert} is therefore smallest when $V(x)=-\sum_{k=r+1}^n \frac1k$. It therefore remains to show that
there exists $n_0\in\mathbb N$ such that, for $n_0\leq m<r$, we have
\begin{equation}\label{a1}
1-\exp\bigg(-\tilde\pi\Big( -\sum_{k=m}^{r} \frac1k, -\sum_{k=m+1}^{r}\frac1k\Big]\bigg)
\geq \beta m^{-\gamma}r^{\gamma-1},
\end{equation}
and, for $n_0\leq r<m$, we have
\begin{equation}\label{a2}
1-\exp\bigg(-\tilde\pi\Big(\sum_{k=r+1}^{m-1} \frac1k, \sum_{k=r+1}^{m} \frac1k\Big]\bigg)
\geq \beta m^{\gamma-1}r^{-\gamma}.
\end{equation}
For \eqref{a1} let $\beta<\beta'<\tilde\beta$ and hence
\begin{align*}
    \tilde\pi\Big( -\sum_{k=m}^{r} \frac1k, -\sum_{k=m+1}^{r}\frac1k\Big] & = \frac{\tilde\beta}{1-\gamma}\exp\Big({-(1-\gamma)\sum_{k=m+1}^{r} \frac1k} \Big)(  e^{\frac{1-\gamma}{m+1}}-1)\\
    & \geq  \frac{\tilde\beta}{m+1} \exp\Big(-(1-\gamma)
    (\log(\tfrac{r}m) + \tfrac{C}{n_0}) \Big)\\[2mm]
    & \geq  \beta'  m^{-\gamma}r^{\gamma-1},
\end{align*}
for some constant $C>0$, if $n_0\leq m<r$ for a suitable $n_0\in\mathbb N$.
Hence, using that $1-e^{-x}\ge (\beta/\beta')x$ for sufficiently small~$x$, we get
\begin{align*}
  1-\exp\bigg(-\tilde\pi\Big( -\sum_{k=m}^{r} \frac1k, -\sum_{k=m+1}^{r}\frac1k\Big]\bigg) & \geq  \beta  m^{-\gamma}r^{\gamma-1}.
\end{align*}
The calculation giving \eqref{a2} is analogous.
\end{proof}


\begin{proof}[Proof of Lemma~\ref{cou1}$(ii)$] For $\beta<\beta'<\beta''<\tilde\beta$ we have \begin{align*}    \tilde\pi\Big( -\infty, -\sum_{k=n_0+1}^{r}\frac1k\Big] & = \frac{\tilde\beta}{1-\gamma}\exp\Big({-(1-\gamma)\sum_{k=n_0+1}^{r} \frac1k} \Big)\\   & \geq  \frac{\tilde\beta}{1-\gamma}   \exp\Big(-(1-\gamma)   (\log(\tfrac{r}{n_0}) + \tfrac{C}{n_0}) \Big)\\[2mm]    & \geq  \frac{\beta''}{1-\gamma}   n_0^{1-\gamma}r^{\gamma-1} \geq \beta' \sum_{m=1}^{n_0} m^{-\gamma}r^{\gamma-1}, \end{align*} for some constant $C>0$ and a suitable $n_0\in\mathbb N$. Hence $$1-\exp\bigg(-\tilde\pi\Big( -\infty, -\sum_{k=n_0+1}^{r}\frac1k\Big]\bigg) \geq 1- \prod_{m=1}^{n_0} e^{-\beta' m^{-\gamma}r^{\gamma-1}}.$$ Using that $e^{-x}\le 1-(\beta/\beta')x$ for sufficiently small~$x$ the result follows. \end{proof}

\subsection{Proof of Proposition~\ref{ldp}}
\label{3.2}

This section is concerned with the large deviations results for the killed branching random walk. The results here are rough versions of the very fine asymptotic results  presented in~\cite{aidekon}. The key difference is that our branching random walk has an infinite offspring distribution so that the moment requirements that are crucially used in 
\cite{aidekon} are not satisfied here. Instead we use  the moment bound on 
\begin{equation}
\label{critmart}
W_n=\sum_{|x|=n} e^{-\rho_- V(x)},
\end{equation}
which we received in Lemmas~\ref{l1} and~\ref{l2} exploiting the Poisson property of the offspring distribution.
Recall that $\mathbb P_u, \mathbb E_u$ refer to initial particles in position $\log u$. Shifting $\log u$ to the origin in Lemma~\ref{l2} we get, for
$1<p<\rho_+/\rho_-$ that
\begin{equation}\label{l2'}
\mathbb E_u \big[W_n^p\big]
\leq C u^{-p\rho_-},
\end{equation}
for some constant $C>0$.
%
We now look at several generations and allow the starting point to be a uniform random variable, note that the expectation $\mathbb E$ refers to exactly this situation.

\begin{lemma}\label{l3}
For $1<p<\rho_+/\rho_-$ we have, for all $N\in\mathbb N$,
that $$\mathbb E \Big[ \Big( \sum_{n=1}^N W_n \Big)^p\Big]
\leq C N^{p+1}.$$
\end{lemma}

\begin{proof}
We first estimate
$$\Big( \sum_{n=1}^N W_n \Big)^p \leq 
N^p  \max_{n=1}^N W_n^p \leq
N^p \sum_{n=1}^N W_n^p.$$
By~\eqref{l2'} we infer from this that
$$\mathbb E_u\Big[ \Big( \sum_{n=1}^N W_n \Big)^p \Big]
\leq N^p \sum_{n=1}^N \mathbb E_u\big[W_n^p\big]
\leq C N^{p+1}  u^{-p\rho_-}.$$
Now take $u$ uniformly random, split according to its value and apply the above, to~get
\begin{align*}
\mathbb E \Big[ \Big( \sum_{n=1}^N W_n \Big)^p\Big]
& \leq \sum_{j=0}^\infty \mathbb P\big( j\leq -\log U < j+1\big) \mathbb E_{e^{-j-1}}\Big[ \Big( \sum_{n=1}^N W_n \Big)^p \Big]\\
& \leq C N^{p+1} \sum_{j=0}^\infty e^{-j+p\rho_-(j+1)} \\
& \leq C N^{p+1} \frac{e^{-p\rho_-} }{1-e^{p\rho_--1}},
\end{align*}
as $p\rho_-<1$. This completes the proof.
\end{proof}

The next auxiliary lemma we need for the proof of Proposition~\ref{ldp} is an easy large deviations bound for the position of the leftmost particle in the branching random walk. Recall that $t^*\in(\rho_-, \rho_+)$ is uniquely defined as the solution of
$$ \frac1{t^*} \log \psi(t^*) = \frac{\psi'(t^*)}{\psi(t^*)}.$$

\begin{lemma}\label{l4}
For every $0<\delta<-\frac{\psi'(t^*)}{\psi(t^*)}$ there exists
$I(\delta)>0$ such that
$$\mathbb P_1 \big( \min_{|x|=N} V(x) \leq \delta N \big) \leq e^{-N I(\delta)}.$$
\end{lemma}

\begin{proof}
We pick $\rho_-<t<t^*$ and by strict convexity of $\log\psi$ we get
$$ \frac1{t} \log \psi(t) < \frac{\psi'(t^*)}{\psi(t^*)}.$$
Now we use the exponential Chebyshev inequality to get
\begin{align*}
\mathbb P_1 \big( \min_{|x|=N} V(x) \leq \delta N \big) & \leq \mathbb P_1 \Big( e^{-t \min\limits_{|x|=N} V(x)} \geq 
e^{-t\delta N} \Big)\\
& \leq e^{t\delta N} 
\mathbb E_1 \Big[ \sum_{|x|=N} e^{-t  V(x)} \Big]  =\exp\big( tN(\delta+ \tfrac1t \log \psi(t) \big)\\
& \leq \exp\bigg( tN\Big(\delta+ \frac{\psi'(t^*)}{\psi(t^*)} \Big)\bigg).
\end{align*}
Now we let $I(\delta)=-\delta-\frac{\psi'(t^*)}{\psi(t^*)}
>0$ to get the desired result.
\end{proof}

Let $\tilde{W}_n$ be as in \eqref{critmart} but with the sum restricted to the particles of the killed branching random walk. Combining the last two lemmas gives us the main step in the proof of Proposition~\ref{ldp}.

\begin{lemma}\label{l5}
We have, for every $\epsilon>0$ and all sufficiently large~$n$,
$$\mathbb P  \Big( \sum_{k=0}^{\infty} \tilde{W}_k \ge n^{\rho_-+\epsilon} \Big) \leq n^{-\rho_++\epsilon}.$$
\end{lemma}

\begin{proof}
We fix $N\in\mathbb N$ arbitrarily and determine the value we require later. Then we split the left-hand side according to survival up to generation~$N$. This yields
\begin{align*}
\mathbb P  \Big( \sum_{k=0}^{\infty} \tilde{W}_k \ge n^{\rho_-+\epsilon} \Big)
& \leq \mathbb P  \Big(\sum_{k=0}^{N-1} W_k \ge n^{\rho_-+\epsilon} \Big)
+ \mathbb P  \big( \tilde{W}_N>0 \big) \\
& \leq n^{-p(\rho_-+\epsilon)} 
\mathbb E \Big[  \Big( \sum_{k=0}^{N-1} W_k \Big)^p \Big]
+ \mathbb P  \big( \min_{|x|=N} V(x) \leq 0 \big) \\
& \leq CN^{p+1} \, n^{-p(\rho_-+\epsilon)} 
+ \mathbb P_1  \big( \min_{|x|=N} V(x) \leq X \big),
\end{align*}
using Lemma~\ref{l3} in the last step. We pick
$0<\delta<-\frac{\psi'(t^*)}{\psi(t^*)}$ and get, by Lemma~\ref{l4}, 
\begin{align*}
\mathbb P_1  \big( \min_{|x|=N} V(x) \leq X \big)
& \leq \mathbb P_1  \big( \min_{|x|=N} V(x) \leq \delta N \big)
+ \mathbb P_1  \big( X \geq \delta N\big) 
\leq e^{-N (I(\delta)+\delta)}. 
\end{align*}
Setting $N =\lceil (\log n) \frac{\rho_+}{I(\delta)+\delta} \rceil$  completes the proof.
\end{proof}

\begin{proof}[Proof of Proposition~\ref{ldp}$(a)$]
Let $\mathscr Z_k$
be the particles alive in the $k$th generation in the killed
branching random walk. Then if there is $x \in \mathscr Z_k$ with $V(x) \leq  -(1-\epsilon)\log n$ we have
$\tilde{W}_k \ge e^{\rho_- (1-\epsilon)\log n}=n^{\rho_- (1-\epsilon)}$. Hence
\begin{align*}
    \mathbb P\big( \exists x \in \mathscr T_{0,1}(-X)  & \text{ with } V(x) \leq  -(1-\epsilon)\log n \big)
     \leq  \mathbb P\Big( \sum_{k=0}^\infty \tilde{W}_k \geq n^{\rho_- (1-\epsilon)} \Big),
\end{align*}
so that the required bound holds by Lemma~\ref{l5}.
\end{proof}

\begin{proof}[Proof of Proposition~\ref{ldp}$(b)$]
We first replace the total population size of the killed branching random by the sum of weighted particles with the same starting point,
$$|\mathscr T_{0,1}(-X)| \leq \sum_{k=0}^\infty \tilde{W}_k,$$
using that in the sum defining $W_n$ all particles located to the left of the origin get weight at least one. Hence
\begin{align*}
\mathbb P  \big( |\mathscr T_{0,1}(-X)| \ge n^{\rho_-+\epsilon} \big)
& \leq \mathbb P  \Big( \sum_{k=0}^{\infty} \tilde{W}_k \ge n^{\rho_-+\epsilon} \Big),
\end{align*}
and again 
the required bound holds by Lemma~\ref{l5}.
\end{proof}




\bibliographystyle{alpha}
\bibliography{mainbib}

\begin{thebibliography}{AHZ13}

\bibitem[AHZ13]{aidekon}
Elie Aïdékon, Yueyun Hu, and Olivier Zindy.
\newblock The precise tail behavior of the total progeny of a killed branching random walk.
\newblock {\em Ann. Probab.}, 41:3786--3878, 2013.

\bibitem[Big77]{biggins}
John~D. Biggins.
\newblock Martingale convergence in the branching random walk.
\newblock {\em J.~Appl. Probab.}, 14:25--37, 1977.

\bibitem[BJR07]{Bollob_s_2007}
Béla Bollobás, Svante Janson, and Oliver Riordan.
\newblock The phase transition in inhomogeneous random graphs.
\newblock {\em Random Struct. Algorithms}, 31:3–122, 2007.

\bibitem[DM13]{Dereich_2013}
Steffen Dereich and Peter Mörters.
\newblock Random networks with sublinear preferential attachment: The giant component.
\newblock {\em Ann. Probab.}, 41:329--384, 2013.

\bibitem[ILL19]{Iksanov}
Alexander Iksanov, Xingang Liang, and Quansheng Liu.
\newblock On {Lp}-convergence of the {Biggins} martingale with complex parameter.
\newblock {\em J. Math. Anal. Appl.}, 479:1653--1669, 2019.

\bibitem[Jan08]{janson}
Svante Janson.
\newblock {The largest component in a subcritical random graph with a power law degree distribution}.
\newblock {\em Ann. Appl. Probab.}, 18:1651 -- 1668, 2008.

\bibitem[Kyp00]{Kyp}
Andreas Kyprianou.
\newblock Martingale convergence and the stopped branching random walk.
\newblock {\em Probab Theory Relat Fields}, 116:405–419, 2000.

\bibitem[LP18]{PPP}
G{\"u}nter Last and Mathew Penrose.
\newblock {\em Lectures on the {Poisson} process}, volume~7 of {\em IMS Textb.}
\newblock Cambridge: Cambridge University Press, 2018.

\bibitem[Lyo97]{lyons}
Russell Lyons.
\newblock A simple path to {Biggins'} martingale convergence for branching random walk.
\newblock In Krishna~B. Athreya and Peter Jagers, editors, {\em Classical and Modern Branching Processes}, pages 217--221, New York, NY, 1997. Springer.

\bibitem[M{\"o}r22]{moerters}
Peter M{\"o}rters.
\newblock Lecture notes on random graphs, 2022.
\newblock Available at https://www.mi.uni-koeln.de/~moerters/lectures/RandomGraphs.pdf.

\bibitem[MS24]{morters_et_al:LIPIcs.AofA.2024.14}
Peter M\"{o}rters and Nick Schleicher.
\newblock Early typical vertices in subcritical random graphs of preferential attachment type.
\newblock In C\'{e}cile Mailler and Sebastian Wild, editors, {\em AofA 2024}, volume 302 of {\em LIPIcs}, pages 14:1--14:10, 2024.

\bibitem[Ner81]{nerman}
Olle Nerman.
\newblock On the convergence of supercritical general ({C}-{M}-{J}) branching processes.
\newblock {\em Z. Wahrscheinlichkeitstheor. Verw. Geb.}, 57:365--395, 1981.

\bibitem[Ray24]{Ray}
Rounak Ray.
\newblock {\em Stochastic processes on preferential attachment models: Understanding global structures from local properties}.
\newblock {PhD Thesis, Eindhoven University of Technology}, Mathematics and Computer Science, 2024.

\bibitem[Shi16]{shi}
Zhan Shi.
\newblock {\em Branching Random Walks: {\'E}cole d'{\'E}t{\'e} de Probabilit{\'e}s de Saint-Flour XLII -- 2012}.
\newblock Lecture Notes in Mathematics. Springer, 2016.

\end{thebibliography}
	
\end{document}